\documentclass[11pt, a4paper]{amsart}
\usepackage{amsmath}
\usepackage{amsfonts}
\usepackage{amssymb}
\usepackage{amsthm}
\usepackage[utf8]{inputenc}

\usepackage{url}
\usepackage{xcolor}

\newtheorem{theorem}{Theorem}[section]
\newtheorem{lemma}[theorem]{Lemma}
\newtheorem{corollary}[theorem]{Corollary}
\newtheorem{proposition}[theorem]{Proposition}
\newtheorem{conjecture}[theorem]{Conjecture}
\newtheorem{problem}[theorem]{Problem}
\def\C{\mathbb{C}}
\def\Q{\mathbb{Q}}
\def\R{\mathbb{R}}
\def\Z{\mathbb{Z}}

\numberwithin{equation}{section}

\title{Extendable orthogonal sets of integral vectors}

\author{Fernando Chamizo}
\thanks{Partially  supported  by  the PID2020-113350GB-I00  grant of the MICINN (Spain) and by ``Severo Ochoa Programme for Centres of Excellence in R{\&}D'' (SEV-2015-0554). This latter grant supported the visit of the second author to the ICMAT where this
work was completed.}
\address{Departamento de Matem\'aticas and ICMAT\\
Universidad Aut\'onoma de Ma\-drid\\
28049 Madrid, Spain}
\email{fernando.chamizo@uam.es}
\author{Jorge Jim\'enez Urroz}
\thanks{Partially supported by the PID2019-110224RB-I00 Grant of the MICINN (Spain).}
\address{Departamento de Matem\'aticas \\
Universitat Polit\`ecnica Catalunya\\
Barce\-lona, Spain}
\email{jorge.urroz@upc.edu}

 

\begin{document}

\begin{abstract}
Motivated by a model in quantum computation we study orthogonal sets of integral vectors of the same norm that can be extended with new vectors keeping the norm and the orthogonality. Our approach involves some arithmetic properties of the quaternions and other hypercomplex numbers.
\end{abstract}

\maketitle


\section{The quantum model and the mathematical problem}

In \cite{GaLa} it is introduced a model of discrete quantum computation that leads to a curious arithmetic problem related to the representation as a sum of squares.
The motivation in that paper is to find a discrete set of states as small as possible that is closed by the Hadamard gate $H$ (inducing entanglement) and by the phase shift gate $\varphi=\pi/2$ (commonly called $S$) with two control qubits. It turns out that this discrete set is the lattice generated by the canonical basis with Gaussian integer coefficients.
In mathematical terms this is related to the nearly tautological fact that the smallest ring containing~$1$ (as an integer) and $i=\sqrt{-1}$ is $\Z[i]$ because $H$ introduces the addition and the subtraction and $S$ the multiplication by~$i$.

Given an orthogonal set of these discrete states, it is always possible to complete it to an orthogonal basis of the vector space, in particular an observable can be constructed having them as eigenstates. The normalization ruins in some way the discreteness and we would like to have all the basis vectors sharing the same norm to clear denominators after the normalization.
In the context of quantum computation the underlying vector space is $\bigotimes_{k=1}^n \C^2$ with $n$ the number of qubits and then the dimension is always a power of two.  It is indeed doubled when we consider the lattice of the discrete set of states above over $\Z$ because $[\Z[i]:\Z]=2$. But we can pose the problem in any dimension from a mathematical point of view. Namely, we consider the following statement:
\smallskip

{\begin{problem}\label{problem}
{\sl For a given dimension $d$ decide whether every set of orthogonal vectors in $\Z^d$ with the same norm can be extended with new integral vectors of the same norm  to get an orthogonal basis of $\R^d$.}
\end{problem}}
\medskip

In \cite{LaGa} the problem is solved in the affirmative for $d=4$ for sets of vectors having as norm the square root of a prime number, meaning that any of these sets can be extended to complete a basis. It is also proved that for any $d>2$ with $4\nmid d$ there are sets such that the extension to an orthogonal basis is not possible. Note that the case $d=2$ is trivial because $(a,b)$ and $(b,-a)$ are orthogonal of the same norm.
\smallskip

To give some insight about the complexity of the situation, we mention some examples {for $d=3$}.
The vector $\vec{v}=(1,3,5)\in\Z^3$ has norm $\sqrt{35}$ and there does not exist any other vector in $\Z^3$ orthogonal to $\vec{v}$ with this norm.
On the other hand, if the starting vector is $\vec{v}=(2,3,6)\in\Z^3$, having norm~$7$, we can complete it to the orthogonal basis
$\big\{\vec{v}, (3,-6,2), (6,2,-3)\big\}$
of vectors of the same norm. An intermediate example is
$\vec{v}= (1, 4, 10)$ with~$\|\vec{v}\|=\sqrt{117}$ which can be extended to
$\big\{\vec{v}, (-8, 7, -2)\big\}$
preserving the norm but it is not possible to extend this  orthogonal set (or any other containing $\vec{v}$) to an orthogonal basis formed by vectors of norm~$\sqrt{117}$.

\medskip

These and other examples {in different dimensions} suggest a finer formulation of the problem above separating different norms and allowing partial extensions. With this idea in mind, we introduce some notation.
Let $\mathcal{O}_d(N,n)$ be the collection of sets $\mathcal{S}\subset\Z^d$ of orthogonal vectors of norm $\sqrt{N}$ with $\#\mathcal{S}=n$. We are interested in characterizing 
\[
 \mathcal{C}_d(n_1,n_2)
 =
 \big\{
 N\in\Z^+\,:\,
 \forall
 \mathcal{S}\in \mathcal{O}_d(N,n_1)
 \ \exists
 \mathcal{S}'\in \mathcal{O}_d(N,n_2)
 \text{ with }
 \mathcal{S}'\supset \mathcal{S}
 \big\}
\]
for $1\le n_1<n_2\le d$.

\medskip

{In the case $d=4$ the aforementioned result of \cite{LaGa} can be rephrased with this notation saying that, for any $1\le n_1<n_2\le 4$,  $\mathcal{C}_d(n_1,n_2)$ contains the prime numbers. However,} the main conjecture in that paper is that the prime numbers do not play any role in the problem for $d=4$:

\smallskip

{\begin{conjecture} \cite{LaGa}\label{conjecture}.
{\sl Problem \ref{problem} has an affirmative answer when $d=4$. In other words, $\mathcal{C}_4(n_1,n_2)=\Z^+$ for $1\le n_1<n_2\le 4$.}
\end{conjecture}}
\medskip

 One of our main results is the full proof of this conjecture. {For instance,
$\vec{v}= (4, 5, 6, 7)$
and $\vec{w}= (-7, -2, -3, 8)$
are vectors in $\Z^4$ with norm $\sqrt{126}$
and we find that adding the vectors
$(-5, -4, 9, -2)$
and
$(-6,  9,  0, -3)$,
having also norm $\sqrt{126}$, we get an orthogonal basis.}

\medskip

The structure of the paper is as follows. In \S\ref{scod1} we treat the case of sets $\mathcal{S}\subset  \mathcal{O}_d(N,d-1)$. It is proved (Corollary~\ref{dminus1}, Corollary~\ref{odddminus1}) that in this situation the extension is always possible if $d$ is even and only possible for integral norms if $d$ is odd. After a digression to introduce some arithmetic results on quaternions in \S\ref{squat}, we face the cases of dimension~$3$ and~$4$ in \S\ref{sdim3}
and
\S\ref{sdim4}.
The solution in the latter case is complete (Theorem~\ref{case4}) and we prove the main conjecture in \cite{LaGa}. For $d=3$ we get a solution of the problem with the original statement (Theorem~\ref{case3sq}) but we fail to completely characterize the norms such that for any vector of that norm there exists a new orthogonal vector of the same norm.
We show (Proposition~\ref{curious}) that these norms when squared are numbers representable as a sum of two squares, in particular they have vanishing asymptotic density. The numerical computations suggest that they form a much more sparse sequence. Finally, in \S6 we prove some results for the higher dimensional cases.

\section{The case of codimension~$1$}\label{scod1}

It will be convenient to consider the following linear algebra result, which is related to the so called matrix determinant lemma \cite[Lemma 1.1]{DiZh} but we have not found it in the literature.

\begin{proposition}\label{linalg}
 Let $A$ be the $(d-1)\times d$ matrix formed by the first $d-1$ rows of an orthogonal matrix. Let $\vec{c}$ be the first column of $A$ and $B$ the rest of the matrix.
 Then
 $|\det(B)|^2= 1-\|\vec{c}\|^2$.
\end{proposition}

\begin{proof}
 Since $AA^t=I$, with $I$ the $(d-1)$-identity matrix, we have
 $\vec{c}\vec{c}^t+BB^t=I$ and the following identities {hold} with $\vec{0}$ the null vector in $\R^{d-1}$:
 \[
  \begin{pmatrix}
   \vec{c}&B
   \\
   1&\vec{0}^t
  \end{pmatrix}
  \begin{pmatrix}
   \vec{c}^t&1
   \\
   B^t&\vec{0}
  \end{pmatrix}
  =
  \begin{pmatrix}
   I&\vec{c}
   \\
   \vec{c}^t&1
  \end{pmatrix}
  =
  \begin{pmatrix}
   I&\vec{0}
   \\
   \vec{c}^t&1
  \end{pmatrix}
  \begin{pmatrix}
   I&\vec{c}
   \\
   \vec{0}^t&1-\|\vec{c}\|^2
  \end{pmatrix}.
 \]
 Taking determinants
 $(-1)^{d-1}\det(B)\cdot (-1)^{d-1}\det(B^t)=1\cdot (1-\|\vec{c}\|^2)$ and the result follows.
\end{proof}

\begin{corollary}\label{detnorm}
 Let $M$ be a $(d-1)\times d$ integral matrix such that its rows form a set in $\mathcal{O}_d(N,d-1)$. Let $\vec{c}_j$ be its $j$-th column and $M_{(j)}$ the resulting square matrix when it is omitted. Then
 $
  |\det(M_{(j)})|^2= N^{d-2}(N-\|\vec{c}_j\|^2).
 $
\end{corollary}

\begin{proof}
 For $j=1$ apply Proposition~\ref{linalg} to $A=N^{-1/2} M$ to get the formula
 $(N^{-1/2})^{2(d-1)}|\det(M_{(j)})|^2= 1-\|\vec{c}_1\|^2/N$, which gives the result. For $j\ne 1$ the same argument works permuting the columns.
\end{proof}

Note that for $d$ even this implies that $N-\|\vec{c}_j\|^2$ is always a square, which does not seem obvious at all. Also, for the first nontrivial odd case  $d=3$, we would have to prove that given $(a,b,c)\in\Z^3$ with $a^2+b^2+c^2=N$, any solution of the Diophantine equation
\[
 \begin{cases}
  x^2+y^2+z^2=N \\ ax+by+cz=0
 \end{cases}
\]
verifies that $N(N-a^2-x^2)$ is a square. Proving it without using the previous result is a challenge harder than it seems. The shortest proof, based only on direct algebraic manipulations, that we have found is:
\begin{align*}
 N(N-a^2-x^2)
 &=(N-a^2)N-Nx^2\\
 & =
 (b^2+c^2)(x^2+y^2+z^2)-(a^2+b^2+c^2)x^2\\
 &=
 -a^2x^2+(b^2+c^2)(y^2+z^2).
\end{align*}
By the second equation in the Diophantine system, $a^2x^2=(by+cz)^2$. Thus
\begin{align*}
 N(N-a^2-x^2)
 &=-(by+cz)^2+b^2y^2+c^2z^2+b^2z^2+c^2y^2\\
 & =b^2z^2+c^2y^2-2bycz=(bz-cy)^2.
\end{align*}

\medskip

The next proposition is a generalization of a result discussed in \cite[\S2]{LaGa} for $d=4$. Note that the proof given there for that particular dimension, by ``polynomial checking'', requires cumbersome calculations, difficult to check without a computer based algebraic manipulator and it is unclear how to generalize the procedure.

\begin{proposition}\label{dminusw}
 With the notation as in Corollary~\ref{detnorm}, for $d$ even the row vector $\vec{w}$ with coordinates $w_j=(-1)^j N^{(2-d)/2}\det(M_{(j)})$ satisfies $\vec{w}\in \Z^d$, is orthogonal to the rest of the rows of $M$ and has the same norm $\sqrt{N}$.
\end{proposition}

\begin{proof}
 The dot product of $\vec{w}$ and the $i$-th row of $M$ is proportional to the sum
 $\sum_{j=1}^d(-1)^{i+j} m_{ij} \det(M_{(j)})$ and it vanishes because it is the expansion along the first row of $M$ completed to a square matrix repeating the $i$-th row as first row. Then $w_j$ is orthogonal to the rows of $M$.

 Using Corollary~\ref{detnorm},
 \[
  \|\vec{w}\|^2
  =
  N^{2-d}
  \sum_{j=1}^d
  N^{d-2}(N-\|\vec{c}_j\|^2)
  =
  dN
  -
  \sum_{j=1}^d
  \|\vec{c}_j\|^2
  =
  dN-(d-1)N =N.
 \]
 Then $\vec{w}$ has the same norm as the rows of $M$.

 Finally, note that Corollary~\ref{detnorm} assures that
 $N^{d-2}$ divides $|\det(M_{(j)})|^2$. If $d$ is even it implies $N^{(d-2)/2}\mid \det(M_{(j)})$, hence $w_j\in\Z$.
\end{proof}
\smallskip
{The previous proposition gives us directly the following}
\begin{corollary}\label{dminus1}
 If $d$ is even then $\mathcal{C}_d(d-1,d)=\Z^+$.
\end{corollary}

\medskip

The odd case can also be treated with similar tools.

\begin{proposition}\label{odddminus1p}
 If $d$ is odd and $N$ is not a square, then
 $\mathcal{S}\in\mathcal{O}_d(N,d-1)$
 cannot be extended to
 $\mathcal{S}'\in\mathcal{O}_d(N,d)$.
\end{proposition}

\begin{proof}
 By definition, $\mathcal{S}\in\mathcal{O}_d(N,d-1)$ spans a subspace $V$ of dimension $d-1$ in $\R^d$. Its orthogonal  complement $V^\perp$ is spanned by $\vec{w}\ne\vec{0}$ in Proposition~\ref{dminusw} and $\vec{w}$ and $-\vec{w}$ are the only vectors with norm $\sqrt{N}$ in $V$, because the proof of Proposition~\ref{dminusw}  only appeals to the parity of $d$ in the final divisibility condition. Hence it is enough to note that $\vec{w}\not\in\Z^d$, which is obvious because $\det(M_{(j)})\in \Z$ and $N^{(2-d)/2}\not\in\Q$.
\end{proof}

\begin{corollary}\label{odddminus1}
 If $d$ is odd then $\mathcal{C}_d(d-1,d)=\{n^2\,:\, n\in\Z^+\}$.
\end{corollary}

\begin{proof}
 The only addition to
 Proposition~\ref{odddminus1p} is that $N^{(d-2)/2}\in\Z$ if $N$ is a square  and the proof of
 Proposition~\ref{dminusw} applies.
\end{proof}

\section{Some considerations about quaternions}\label{squat}

The cases $d=3$ and $d=4$ are treated using the Hamilton quaternions $\mathcal{H}$. The purpose of this section is to introduce some notation and state some results for later reference. Recall that $\mathcal{H}$ is composed by expressions of the form
\[
 \mathbf{q}
 =
 a_0
 +a_1\mathbf{i}
 +a_2\mathbf{j}
 +a_3\mathbf{k}
 \qquad\text{with }a_j\in\R
\]
and $(\mathcal{H},+,\cdot)$ becomes an associative normed division algebra over $\R$ imposing $\mathbf{i}^2=\mathbf{j}^2=\mathbf{k}^2=\mathbf{i}\mathbf{j}\mathbf{k}=-1$ and with the squared norm $\|\mathbf{q}\|^2=\sum_{j=0}^3a_j^2=\mathbf{q}\overline{\mathbf{q}}$ where $\overline{\mathbf{q}}$ is the conjugate quaternion
$ a_0
 -a_1\mathbf{i}
 -a_2\mathbf{j}
-a_3\mathbf{k}$.
It is important to keep in mind $\overline{\mathbf{q}_1\mathbf{q}_2}=\overline{\mathbf{q}_2}\,\, \overline{\mathbf{q}_1}$.

A well known theorem due to Frobenius assures that $\mathcal{H}$ is the largest associative division algebra over $\R$ and Hurwitz proved that if we drop the associativity keeping the norm, the only possible extension is the algebra of Cayley numbers (also named octonions). These results impose a limit to extend our approach to higher dimensions.  The book \cite{KaSo} is a nice introduction to these and other topics at an elementary level (see also \cite{stillwell} for the role of~$\mathcal{H}$ as a Lie group). We refer the reader to it for the basic properties of the quaternions.
\smallskip

In our case, we are going to consider only quaternions with $a_j\in\Z$ and
$\mathcal{H}_\Z$
will denote this set. From the algebraic point of view,
$\mathcal{H}_\Z$
is the lattice generated by $\{1,\mathbf{i},\mathbf{j},\mathbf{k}\}$ over $\Z$
and it misses $(1+\mathbf{i}+\mathbf{j}+\mathbf{k})/2$ to constitute a maximal order (Hurwitz's quaternions). It causes a parity issue in some contexts. It is known  that replacing
$\mathcal{H}_\Z$
by
the maximal order we would have unique factorization in a highly non obvious way \cite{CoPe}, \cite[\S5]{CoSm}. We prefer to avoid here any reference to factorization because our results admit proofs without entering into this intricate topic, although our initial approach was partially based on it.
\smallskip

We consider the embedding of $\Z^3$ in $\mathcal{H}_\Z$ given by
\[
 \vec{a}=
 (a_1,a_2,a_3)\in\Z^3
 \lhook\joinrel\xrightarrow{\qquad}
 \mathbf{q}_{\vec{a}}=
 a_1\mathbf{i}
 +a_2\mathbf{j}
 +a_3\mathbf{k}
 \in\mathcal{H}_\Z.
\]
The relation between vectors in dimension $3$ and quaternions is not spurious and in fact the motivation of Hamilton was to find a vector multiplication resembling the complex number product and its relation to rotations.

\begin{lemma}\label{qdotcross}
 We have
 $\mathbf{q}_{\vec{a}}\mathbf{q}_{\vec{b}}
 =-\vec{a}\cdot \vec{b}+\mathbf{q}_{\vec{a}\times\vec{b}}$
 where $\vec{a}\cdot \vec{b}$ and $\vec{a}\times\vec{b}$ are the usual dot and cross products.
\end{lemma}

\begin{proof}
 It reduces to a calculation \cite[\S4.1]{KaSo}.
\end{proof}

\begin{lemma}\label{quatvector}
 If $\vec{a}\in\Z^3$ and $\mathbf{q}\in\mathcal{H}_{\Z}$, then $\mathbf{q}\mathbf{q}_{\vec{a}}\overline{\mathbf{q}}=\mathbf{q}_{\vec{b}}$ for some $\vec{b}\in\Z^3$
 and $\|\vec{b}\|=\|\mathbf{q}\|^2\|\vec{a}\|$.
\end{lemma}

In geometric terms, $\vec{b}$ is obtained from $\vec{a}$ after a rotation and a homothety depending on $\mathbf{q}$ \cite[\S3.1]{CoSm}.

\begin{proof}
 Clearly $\mathbf{q}\mathbf{q}_{\vec{a}}\overline{\mathbf{q}}\in\mathcal{H}_{\Z}$
 and the conjugate of $\mathbf{q}_{\vec{a}}$ is $-\mathbf{q}_{\vec{a}}$. Then the conjugate of $\mathbf{q}\mathbf{q}_{\vec{a}}\overline{\mathbf{q}}$ equals its negative and hence its first coordinate vanishes.
 Plainly $\|\mathbf{q}_{\vec{a}}\|=\|\vec{a}\|$ and $\|\vec{b}\|=\|\mathbf{q}\|^2\|\vec{a}\|$ follows since the norm is multiplicative.
\end{proof}

\begin{lemma}\label{qijk}
 Given $\mathbf{q}\in\mathcal{H}_{\Z}$, define $\vec{a},\vec{b},\vec{c}\in\Z^3$
 by
 $\mathbf{q}_{\vec{a}}=\mathbf{q}\mathbf{i}\overline{\mathbf{q}}$,
 $\mathbf{q}_{\vec{b}}=\mathbf{q}\mathbf{j}\overline{\mathbf{q}}$,
 $\mathbf{q}_{\vec{c}}=\mathbf{q}\mathbf{k}\overline{\mathbf{q}}$.
 Then $\{\vec{a},\vec{b},\vec{c}\}$ is an orthogonal set of vectors of the same norm.
\end{lemma}

\begin{proof}
 First of all, note that Lemma~\ref{quatvector} assures that $\vec{a}$, $\vec{b}$ and $\vec{c}$ are well defined and they have the same norm.

 By Lemma~\ref{qdotcross} and a direct calculation,
 $-\vec{a}\cdot \vec{b}+\mathbf{q}_{\vec{a}\times\vec{b}}=\mathbf{q}_{\vec{a}}\mathbf{q}_{\vec{b}}=\|\mathbf{q}\|^2\mathbf{q}\mathbf{k}\overline{\mathbf{q}}=\|\mathbf{q}\|^2\mathbf{q}_{\vec{c}}$ then
 $\vec{a}\cdot\vec{b}=0$. The same argument shows $\vec{b}\cdot\vec{c}=\vec{c}\cdot\vec{a}=0$.
\end{proof}
\smallskip

Following \cite{CoPe}, we say that $\mathbf{q}\in\mathcal{H}_{\Z}$ is primitive if $ \mathbf{q}=m\mathbf{q}'$ with $m\in\Z^+$ and $\mathbf{q}'\in\mathcal{H}_{\Z}$ implies $m=1$.
In other words, if the coefficients of $\mathbf{q}$ have not a nontrivial common factor.

The next result is just a synthetic form of writing the parametrization of the Pythagorean quadruples.

\begin{proposition}\label{quatpyth}
 If $\vec{a}\in\Z^3$, $\|\vec{a}\|\in\Z^+$ and $\mathbf{q}_{\vec{a}}$ is primitive, there exists $\mathbf{q}\in\mathcal{H}_{\Z}$ such that
 $\mathbf{q}_{\vec{a}}\in \{\mathbf{q}\mathbf{i}\overline{\mathbf{q}}, \mathbf{q}\mathbf{j}\overline{\mathbf{q}}, \mathbf{q}\mathbf{k}\overline{\mathbf{q}}\}$.
\end{proposition}

\begin{proof}
 Expanding the products we have the parametrizations of the Pytha\-gorean quadruple $(a_1,a_2,a_3,\|\vec{a}\|)$.
 See the details in \cite{cremona} and \cite{spira}.
\end{proof}

Passing to the maximal order, the following result could be rephrased saying that couples like $\mathbf{q}$ and $\mathbf{q}\mathbf{i}$ are coprime to the right when $\mathbf{q}\mathbf{i}\overline{\mathbf{q}}$ is primitive although they share the factor $\mathbf{q}$ to the left. It will be crucial in our solution of the problem for $d=4$.

\begin{proposition}\label{quatgcd}
 Let $\mathbf{u}\in\{\mathbf{i}, \mathbf{j}, \mathbf{k}\}$
 and $\mathbf{q}\in\mathcal{H}_{\Z}$.
 If $\mathbf{q}\mathbf{u}\overline{\mathbf{q}}$ is primitive then there exist $\mathbf{q}_1,\mathbf{q}_2\in\mathcal{H}_{\Z}$ such that
 $\mathbf{q}_1\mathbf{q}+\mathbf{q}_2\mathbf{q}\mathbf{u}=2$.
\end{proposition}

\begin{proof}
 Let $\mathbf{q}=a+b\mathbf{i}
 +c\mathbf{j}
 +d\mathbf{k}$. A calculation shows
 \[
  \mathbf{q}-\mathbf{i}\mathbf{q}\mathbf{i} = 2(a+b\mathbf{i})
  \qquad\text{and}\qquad
  -\mathbf{j}\mathbf{q}+\mathbf{k}\mathbf{q}\mathbf{i} = 2(c-d\mathbf{i}).
 \]
 The Gaussian integers are embedded in $\mathcal{H}_\Z$ via $i\mapsto \mathbf{i}$ preserving the operations of the algebra $\mathcal{H}$ \cite[\S6.1]{KaSo}.
 Assuming that $a+bi$ and $c-di$ are coprime Gaussian integers, the Euclidean algorithm gives $A,B,C,D\in\Z$ such that
 \[
  (A+B\mathbf{i})(\mathbf{q}-\mathbf{i}\mathbf{q}\mathbf{i})
  +
  (C-D\mathbf{i})(-\mathbf{j}\mathbf{q}+\mathbf{k}\mathbf{q}\mathbf{i})
  = 2
 \]
 and it would prove the result for $\mathbf{u}=\mathbf{i}$ with
 $\mathbf{q}_1= (A+B\mathbf{i}) -(C-D\mathbf{i})\mathbf{j}$
 and
 $\mathbf{q}_2= (B-A\mathbf{i}) +(C-D\mathbf{i})\mathbf{k}$.

 Let us see that the existence of a Gaussian prime dividing $a+bi$ and $c-di$ leads to a contradiction. Let $p$ be the rational prime over it i.e., $(p)$ is the prime ideal in $\Z$ determined by the integers divisible by the Gaussian prime. Then
 \[
  p\mid (a+bi)(a-bi)=a^2+b^2,
  \quad
  p\mid (c-di)(c+di)=c^2+d^2,
  \quad
  p\mid (a+bi)(c+di).
 \]
 A calculation using Lemma~\ref{qdotcross} shows
 \[
  -\mathbf{q}\mathbf{i}\overline{\mathbf{q}}\mathbf{i}
  =
  a^2+b^2-(c^2+d^2) +2(a+b\mathbf{i})(c+d\mathbf{i})\mathbf{j}.
 \]
 Then $p$ divides  the coefficients of $-\mathbf{q}\mathbf{i}\overline{\mathbf{q}}\mathbf{i}$ and hence those of $\mathbf{q}\mathbf{i}\overline{\mathbf{q}}$, contradicting that it is primitive. This concludes the proof for $\mathbf{u}=\mathbf{i}$.

 Finally, note that any circular permutation of $\mathbf{i}$, $\mathbf{j}$, $\mathbf{k}$ induces a bijective map $C$ on $\mathcal{H}$ preserving the algebra operations,
 $C(\mathbf{q}_1+\mathbf{q}_2)=C(\mathbf{q}_1)+C(\mathbf{q}_2)$,
 $C(\mathbf{q}_1\mathbf{q}_2)=C(\mathbf{q}_1)C(\mathbf{q}_2)$,
 because it preserves the relations defining the algebra. Then the result for
 $\mathbf{u}=\mathbf{i}$ implies it for $\mathbf{j}$ and $\mathbf{k}$.
\end{proof}

\section{The case $d=3$}\label{sdim3}

We start showing that there is a restriction on the factorization of $N$ if we want $\mathcal{O}_3(N,2)$ to be nonempty. Recall the notation $p^\alpha\|N$ with $p^\alpha$ a prime power meaning $p^\alpha\mid N$ and $p^{\alpha+1}\nmid N$.

\begin{proposition}\label{curious}
 If $\{\vec{a},\vec{b}\}\in\mathcal{O}_3(N,2)$ then $N$ is representable as a sum of two squares.
\end{proposition}

\begin{proof}
 Recall \cite{grosswald} that a positive integer $N$ is representable as a sum of two squares if and only if 
 there does not exist a prime number $q\equiv 3\pmod{4}$ such that $q^{2k-1}\| N$ with $k\in\Z^+$.
 
 \medskip
 
 If $q^{2k-1}\| N$, using Corollary~\ref{detnorm} with $d=3$ and $j=1$, we must have $q^{2k-1+2l-1}\| N(N-a_1^2-b_1^2)$ for some $l\in\Z^+$.
 Let us call  $N= q^{2k-1}N'$ and $N-a_1^2-b_1^2=q^{2l-1}M'$ with $q\nmid N',M'$.
 We have
 \[
  q^{2(k-l)}N'-M'
  =
  \frac{a_1^2+b_1^2}{q^{2l-1}}
  \qquad\text{and}\qquad
  N'-q^{2(l-k)}M'
  =
  \frac{a_1^2+b_1^2}{q^{2k-1}}.
 \]
 If $l\ne k$ we deduce $q^{2\min(k,l)-1}\|a_1^2+b_1^2$ and this is a contradiction. If $l=k$ the only way of avoiding this contradiction is $q^{2k}\mid a_1^2+b_1^2$.
 A circular permutation of the coordinates of $\vec{a}$ and $\vec{b}$ preserves the norm and the orthogonality. Hence we have $q^{2k}\mid a_j^2+b_j^2$ for $j=1,2,3$. Adding these divisibility conditions, we get $q^{2k}\mid N+N$  that contradicts $q^{2k-1}\| N$.
\end{proof}

\begin{corollary}\label{case31nr}
 If a prime of the form $4n+3$ appears in the factorization of $N$ with odd exponent then $\mathcal{O}_3(N,n)$ is empty for $n=2,3$.
\end{corollary}

\begin{theorem}\label{case3sq}
 We have
 $\mathcal{C}_3(1,2)\supset\mathcal{C}_3(1,3)=\mathcal{C}_3(2,3)=\{n^2\,:\,n\in\Z^+\}$.
\end{theorem}

\begin{proof}
 The inclusion $\mathcal{C}_3(1,2)\supset\mathcal{C}_3(1,3)$
 is trivial and we know $\mathcal{C}_3(2,3)=\{n^2\,:\,n\in\Z^+\}$ by Corollary~\ref{odddminus1} and  $\mathcal{C}_3(1,3)\subset\{n^2\,:\,n\in\Z^+\}$ by Proposition~\ref{odddminus1p}.
 Then we have to prove that $\mathcal{C}_3(1,3)$ includes the squares. This means that any $\vec{a}=(a_1,a_2,a_3)\in \Z^3$ with $\|\vec{a}|\in\Z^+$
 can be completed with two other vectors in $\Z^3$ of the same norm to get an orthogonal basis.

 It is plain that we can restrict ourselves to the case $\gcd(a_1,a_2,a_3)=1$, equivalently, $\mathbf{q}_{\vec{a}}$ is primitive.
 By Proposition~\ref{quatpyth}
 $\mathbf{q}_{\vec{a}}\in \{\mathbf{q}\mathbf{i}\overline{\mathbf{q}}, \mathbf{q}\mathbf{j}\overline{\mathbf{q}}, \mathbf{q}\mathbf{k}\overline{\mathbf{q}}\}$ for some
 $\mathbf{q}\in\mathcal{H}_{\Z}$.
 Finally, Lemma~\ref{qijk} shows that we can extend $\vec{a}$ to an orthogonal set $\{\vec{a},\vec{b},\vec{c}\}\subset\Z^3$ of vectors of the same norm.
\end{proof}

Although $\mathcal{C}_3(1,2)$ seems to be close to the squares,
some examples show that the inclusion in Theorem~\ref{case3sq} is strict. For instance, $\mathcal{O}_3(18,1)$ is composed by the vectors $(0,3,3)$, $(1,1,4)$ and all the rearrangements and sign changes of their coordinates. We can complete the first vector with $(0,3,-3)$ and the second with $(3,-3,0)$. Then $18\in\mathcal{C}_3(1,2)$ and $18\not\in\mathcal{C}_3(1,3)$ because it is not a square. A more complicated example without repeated absolute values of the coordinates in all vectors occurs for $N=98$. The relevant representations corresponds to $(0,7,7)$, $(1,4,9)$ and $(3,5,8)$, which can be completed with $(0,7,-7)$, $(5, -8,3)$ and $(9,1,-4)$, respectively. Consequently, $98\in \mathcal{C}_3(1,2)$ and $98\not\in\mathcal{C}_3(1,3)$.

We have not been able of characterizing the difference set $\mathcal{C}_3(1,2)\setminus\mathcal{C}_3(1,3)$.
Running a computer program we have got that this set includes:
\[
	\{
	18,
	45,
	50,
	72,
	85,
	90,
	98,
	117,
	125,
	130,
	162,
	180,
	200,
	242,
	245,
	250,
	288, \dots
	\}.
\]
In fact the listed numbers cover all the values with $N<300$ excluding the trivial cases with essentially only a representation as a sum of three squares (there are only a finite number of them \cite{BaGr} being the largest $427$).

\section{The case $d=4$}\label{sdim4}

We now solve in the affirmative the original problem for $d=4$, which is a conjecture in \cite{LaGa}. 

\begin{theorem}\label{case4}
 For any $1\le n_1<n_2\le 4$ we have $\mathcal{C}_4(n_1,n_2)=\Z^+$. 
\end{theorem}

In other words, any set $\mathcal{S}\subset\Z^4-\{\vec{0}\}$ of orthogonal vectors of the same norm can be extended to an orthogonal basis $\mathcal{B}\supset\mathcal{S}$ of $\R^4$ keeping all the basis vectors in $\Z^4$ and with the same norm.

\begin{proof}
 Given a (row) vector $\vec{v}=(v_1,v_2,v_3,v_4)\in\Z^4$, consider the quaternion $\mathbf{v}=v_1+v_2\mathbf{i} +v_3\mathbf{j}+v_4\mathbf{k}$ and the vectors $\vec{a}$, $\vec{b}$ and $\vec{c}$ in $\Z^4$ whose coordinates are given, respectively, by the coefficients of
 $\mathbf{i}\mathbf{v}$, $\mathbf{j}\mathbf{v}$ and $\mathbf{k}\mathbf{v}$.
 Clearly the four vectors have the same norm and to settle the case $n_1=1$ we have to show that they are orthogonal. Using
 $\mathbf{i}^2=\mathbf{j}^2=\mathbf{k}^2=-1$ it is deduced that $\vec{v}\cdot\vec{a}$ is the real part (the first coefficient) of
 $\mathbf{v}\overline{\mathbf{i}\mathbf{v}}$
 and it vanishes because this is $\mathbf{v}\overline{\mathbf{v}}(-\mathbf{i})=-\|\vec{v}\|^2\mathbf{i}$.
 The same argument works to prove the orthogonality of any couple of the vectors.
 \smallskip

 The case $n_1=3$ follows from Corollary~\ref{dminus1} with $d=4$.
 \smallskip

 It remains to consider $n_1=2$, which is the harder case. We already know $\mathcal{C}_3(3,4)=\Z^+$ then we have to prove that
 $\{\vec{v},\vec{w}\}\subset\Z^4$
 with $\vec{v}\cdot\vec{w}=0$, $\|\vec{v}\|=\|\vec{w}\|$,
 can  be extended with a new orthogonal vector $\vec{u}\in\Z^4$ of the same norm. Let $\vec{a}$, $\vec{b}$ and $\vec{c}$ be defined as before.
 They generate in $\Q^4$ the orthogonal subspace to $\vec{v}$ so there exist $\ell_1,\ell_2,\ell_3\in\Z$ and $Q\in\Z^+$ such that
 \[
  \vec{w}=
  \frac{\ell_1}{Q}\vec{a}
  +
  \frac{\ell_2}{Q}\vec{b}
  +
  \frac{\ell_3}{Q}\vec{c}
  \qquad\text{with}\quad
  \ell_1^2+\ell_2^2+\ell_3^2
  =
  Q^2.
 \]
 The last relation follows from $\|\vec{v}\|=\|\vec{w}\|$.
 We can assume $\gcd(\ell_1,\ell_2,\ell_3)=1$ freely because otherwise we could simplify the fractions $\ell_j/Q$. Under this assumption necessarily $Q$ is odd because $\ell_j^2\equiv 0,1\pmod{4}$.

 By Proposition~\ref{quatpyth} applied to $\vec{\ell}=(\ell_1,\ell_2,\ell_3)$ there exists $\mathbf{q}\in\mathcal{H}_\Z$ such that
 $\mathbf{q}_{\vec{\ell}}
 =
 \mathbf{q}\mathbf{u}\overline{\mathbf{q}}$
 with $\mathbf{u}\in\{\mathbf{i},\mathbf{j},\mathbf{k}\}$
 and  by Lemma~\ref{qijk}, we obtain $\vec{k}=(k_1,k_2,k_3)\in \Z^3$ orthogonal to $\vec{\ell}$ defined by
 $\mathbf{q}_{\vec{k}}
 =
 \mathbf{q}\mathbf{u}'\overline{\mathbf{q}}$, $\mathbf{u}'\in\{\mathbf{i},\mathbf{j},\mathbf{k}\}$,
 and  $\|\vec{\ell}\|=\|\vec{k}\|=
 \|\mathbf{q}\|^2=Q$.
 Now we are ready to define $\vec{u}$. We take
 $\vec{u}=(k_1\vec{a}+k_2\vec{b}+k_3\vec{c})/Q$, which is orthogonal to $\vec{v}$ and $\vec{w}$ and of the same norm. The only missing point is to show $\vec{u}\in\Z^4$.

 Note that $Q\vec{w}$ is the vector having as coordinates the coefficients of
 \[
  \ell_1\mathbf{i}\mathbf{v}
  +
  \ell_2\mathbf{j}\mathbf{v}
  +
  \ell_3\mathbf{k}\mathbf{v}
  =
  \mathbf{q}_{\vec{\ell}}\mathbf{v}
  =
  \mathbf{q}\mathbf{u}\overline{\mathbf{q}}\mathbf{v}.
 \]
 In particular $Q$ divides
 $\mathbf{q}\mathbf{u}\overline{\mathbf{q}}\mathbf{v}$. A similar chain of equalities shows that the coordinates of $Q\vec{u}$ are the coefficients of
 $\mathbf{q}\mathbf{u}'\overline{\mathbf{q}}\mathbf{v}$. We are going to prove that $Q$ divides
 $\overline{\mathbf{q}}\mathbf{v}$, which implies $\vec{u}\in\Z^4$. Multiplying to the right by
 $\overline{\mathbf{q}}\mathbf{v}$
 the equation in Proposition~\ref{quatgcd}, we get
 \[
  2\overline{\mathbf{q}}\mathbf{v}
  =
  \mathbf{q}_1\mathbf{q}\overline{\mathbf{q}}\mathbf{v}
  +
  \mathbf{q}_2\mathbf{q}\mathbf{u}\overline{\mathbf{q}}\mathbf{v}
  =
  Q
  \Big(
  \mathbf{q}_1\mathbf{v}
  +
  \mathbf{q}_2\frac{\mathbf{q}\mathbf{u}\overline{\mathbf{q}}\mathbf{v}}{Q}
  \Big)
 \]
 and, as $Q$ is odd, $Q$ must divide the coefficients of
 $\overline{\mathbf{q}}\mathbf{v}$.
\end{proof}

\section{Other dimensions}

Once the problem is solved for $d=4$ and we have proved that in general is impossible to complete an orthogonal basis for $2\nmid d$ (Proposition~\ref{odddminus1p}), the natural question is what happens in the rest of the dimensions. The conjecture in \cite{LaGa} is that the extension is always possible when the dimension is a multiple of four (see \cite[Conj.2]{LaGa}).
With our notation, this is the claim $\mathcal{C}_{d}(n_1,n_2)=\Z^+$ for $1\le n_1<n_2\le d$ when $4\mid d$.

In the even case not covered by the conjecture, $4\mid d-2$, the extension is not possible in general, as shown in the next result, which is in \cite[\S3]{LaGa}. Our proof is essentially the same avoiding the matrix notation.  

\begin{proposition}
 For $d=4k+2$, $k\in\Z^+$,  the numbers not representable as a sum of two squares are not in  $\mathcal{C}_d(d-2,d-1)$. In particular, they are not in 
 $\mathcal{C}_d(d-2,d)$.
\end{proposition}

\begin{proof}
 Let $N\in \mathcal{C}_d(d-2,d-1)$. By Theorem~\ref{case4}, we have 
 $\{\vec{v}_1,\vec{v}_2,\vec{v}_3,\vec{v}_4\}\in\mathcal{O}_d(N,4)$
 having all the coordinates zero except at most those in the  first four places. In general, we can construct 
 $\{\vec{v}_{4j-3},\vec{v}_{4j-2},\vec{v}_{4j-1},\vec{v}_{4j}\}\in\mathcal{O}_d(N,4)$,
 $1\le j\le k$, 
 supported on the coordinates $4j-3$, $4j-2$, $4j-1$ and $4j$. 
 Clearly $\mathcal S=\{\vec{v}_1,\vec{v}_2,\dots,\vec{v}_{4k}\}\in\mathcal{O}_d(N,d-2)$
 and they generate a subspace $V\subset\R^d$ such that $V^\perp$ is the $2$-dimensional subspace formed by the vectors with the first $d-2$ coordinates zero. Then any $\vec{v}$ such that
 $\mathcal{S}\cup\{\vec{v}\}\in\mathcal{O}_d(N,d)$
 must have at most two non zero coordinates and $N=\|\vec{v}\|^2$ is representable as a sum of two squares. 
\end{proof}

\medskip

The case $n_1=1$ of Theorem~\ref{case4} was settled using the left multiplication by the pure quaternion units 
$\mathbf{i}$,
$\mathbf{j}$ and
$\mathbf{k}$.
In dimension~8, something similar can be done with the pure octonion units. For basic information about the Cayley numbers (octonions), we refer the reader to \cite{KaSo}, \cite{CoSm} and \cite{baez}. 
Here we only use that the elements of this non associative normed division algebra can be identified with vectors $\vec{x}=(x_0,\dots,x_7)\in \R^8$ via $\sum_{j=0}^7 x_j\mathbf{u}_j$ where $\{\mathbf{u}_j\}$ are units generating the algebra with $\mathbf{u}_0=1$ and $\mathbf{u}_j^2=-1$. 
The right multiplication by 
$\mathbf{u}_j$ corresponds to the $j+1$ row vector of the following antisymmetric matrix:
\[
 C(\vec{x})= 
	\begin{pmatrix}
	x_0& -x_1& -x_2& -x_3& -x_4& -x_5& -x_6& -x_7 \\ 
	x_1&	x_0   & -x_4 & -x_7 &  x_2 & -x_6 &  x_5 &  x_3 \\
	x_2&	x_4 &  x_0   & -x_5 & -x_1 &  x_3 & -x_7 &  x_6 \\
	x_3&	x_7 &  x_5 & x_0    & -x_6 & -x_2 &  x_4 & -x_1 \\
	x_4&	-x_2 &  x_1 &  x_6 &  x_0   & -x_7 & -x_3 &  x_5 \\
	x_5&	x_6 & -x_3 &  x_2 &  x_7 &  x_0   & -x_1 & -x_4 \\
	x_6&	-x_5 &  x_7 & -x_4 &  x_3 &  x_1 & x_0    & -x_2 \\
	x_7&	-x_3 & -x_6 &  x_1 & -x_5 &  x_4 &  x_2 & x_0
	\end{pmatrix}.
\]
To be more concrete, the product $C(\vec{x})\vec{y}$, with $\vec{y}\in\R^8$ a column vector, gives the coordinates of the Cayley number product 
$\big(\sum_{j=0}^7 x_j\mathbf{u}_j\big)\big(\sum_{j=0}^7 y_j\mathbf{u}_j\big)$.

The rows of $C(\vec{x})$ are orthogonal (cf. \cite[\S8.4]{CoSm}) and we deduce at once:

\begin{proposition}
 We have $\mathcal{C}_8(1,n_2)=\Z^+$ for $1<n_2\le 8$. 
\end{proposition}

Dividing into $8$-blocks and using $C(\vec{x})$, or into $4$-blocks and using the case $n_1=1$ of Theorem~\ref{case4}, there is an immediate conclusion for higher dimensions. 

\begin{corollary}\label{triv84}
 We have $\mathcal{C}_d(1,n_2)=\Z^+$ for $1<n_2\le 8$ if $8\mid d$ and for $1<n_2\le 4$ if $4\mid d$. 
\end{corollary}
\medskip

Given $\vec{v}\in\R^7$, 
let $P(\vec{v})$ be
the submatrix $\big(c_{ij}(\vec{x})\big)_{i,j=2}^8$ with $\vec{x}=(0,\vec{v})$. It corresponds to the ``pure part'' (the part not including $\mathbf{u}_0$) of the product of pure octonions (those with $x_0=0$). By analogy with Lemma~\ref{qdotcross}, it defines a cross product in $\R^7$ with the usual properties \cite[\S7.4]{lounesto}. Namely, for (column) vectors $\vec{v},\vec{w}\in\R^7$ we define 
\[
 \vec{v}\times\vec{w}=P(\vec{v})\vec{w}
\]
and we have 
\begin{equation}\label{crprpr}
 (\vec{v}\times\vec{w})\cdot \vec{v}=0,
 \quad
 (\vec{v}\times\vec{w})\cdot \vec{w}=0
 \quad\text{and}\quad
 \|\vec{v}\times\vec{w}\|^2+(\vec{v}\cdot\vec{w})^2= \|\vec{v}\|^2\|\vec{w}\|^2.
\end{equation}
In fact, it can be proved that a binary cross product can only be defined for $d=3$ and $d=7$ and it has a topological significance \cite{massey}.

If we try to parallel our reasoning for 
$\mathcal{C}_3(2,3)$
to treat
$\mathcal{C}_7(2,3)$,
noting that Proposition~\ref{dminusw} for $d=3$ defines essentially the standard cross product in~$\R^3$, we find a serious obstruction because we lack the divisibility condition deriving from Corollary~\ref{detnorm}. 
We can state anyway a very weak analogue of a part of Theorem~\ref{case3sq}. In the following result, $k\mid \vec{v}$ means that every coordinate of $\vec{v}$ is a multiple of~$k$.

\begin{proposition}
 Let $K_1$, $K_2$ be positive integers and 
 $\{\vec{v},\vec{w}\}\in\mathcal{O}_7(N,2)$
 with $N=K_1^2K_2^2$.
 If $K_1\mid\vec{v}$,  $K_2\mid\vec{w}$
 then there exists a vector $\vec{u}\in\Z^7$
 such that 
 $\{\vec{v},\vec{w},\vec{u}\}\in\mathcal{O}_7(N,3)$.
\end{proposition}

\begin{proof}
 From our hypothesis 
 $\vec{v}_0=\vec{v}/K_1$
 and
 $\vec{w}_0=\vec{w}/K_2$
 are integral orthogonal vectors with norms $K_2$ and $K_1$, respectively. 
 Taking $\vec{u}=\vec{v}_0\times\vec{w}_0$, the properties \eqref{crprpr} assure that $\{\vec{v},\vec{w},\vec{u}\}$ is an orthogonal set and 
 $\|\vec{u}\|^2=\|\vec{v}\|^2=\|\vec{w}\|^2=N$. Clearly $P(\vec{v}_0)$ has integral entries, then $\vec{u}\in\Z^7$.
\end{proof}

For instance, the previous result with $K_1=8$ and $K_2=9$ when applied to the orthogonal vectors
\[
 \vec{v}=
	(  8,   8,  24,  64,   8,   8,  16) 
 \quad\text{and}\quad 
 \vec{w}=
	( -9,   9,   9, -18,  18,  63,  18),
\]
having $\|\vec{v}\|^2=\|\vec{w}\|^2=8^2\cdot 9^2$, gives
\[
 \vec{u}=
	( -1, -13,  53, -20, -30, -11,  28).
\]
Sometimes for two vectors $\vec{v}$ and $\vec{w}$ not fulfilling the divisibility conditions, by chance, we have that  $\vec{v}\times\vec{w}$ is divisible by $\|\vec{v}\|=\sqrt{N}$ and then we can take $\vec{u}=\|\vec{v}\|^{-1}\vec{v}\times\vec{w}$. There are many examples when $N$ is the square of a relatively small number. For instance, the vectors 
\[
 \vec{v}=
	( 1,   1,   8,  17,   1,   1,   2) 
 \quad\text{and}\quad 
 \vec{w}=
	( 3,  -1,  -3,  -1,  -1,   4,  18)
\]
verify $\{\vec{v},\vec{w}\}\in\mathcal{O}_7(361,2)$ and 
\[
 \vec{u}=\frac{1}{19}\vec{v}\times\vec{w}=
	(9,   3,   3,  -1, -16,   1,  -2)
\]
allows to extend the set to $\{\vec{v},\vec{w},\vec{u}\}\in\mathcal{O}_7(361,3)$.
\smallskip

As a matter of fact, apart of the binary cross products in $\R^3$ and $\R^7$ and the cross product of $d-1$ vectors in $\R^d$ (reflected in  Proposition~\ref{dminusw}) there exists also a ternary cross product in $\R^8$. This exhausts all the possibilities for cross products with the usual properties \cite[\S7.5]{lounesto}. 
See \cite[Th.2.1]{zvengrowski} for the expression of this ternary cross product in terms of the Cayley numbers and its properties. With our notation it corresponds to the formula
\[
 \vec{x}\times\vec{y}\times\vec{z}
 =
 -C(\vec{x})C(\vec{y}^*)\vec{z}
 +(\vec{y}\cdot \vec{z})\vec{x}
 -(\vec{z}\cdot \vec{x})\vec{y}
 +(\vec{x}\cdot \vec{y})\vec{z}
\]
where we consider $\vec{z}\in\R^8$ as a column vector to perform the matrix multiplication. Here $\vec{y}^*$ means $(y_0,-y_1,-y_2,\dots, -y_7)$ and the dot indicates the standard inner product in $\R^8$. It is apparent that $\vec{x}\times\vec{y}\times\vec{z}$ works finely as a map $\Z^8\times \Z^8\times \Z^8\longrightarrow\Z^8$. A variant of the previous proposition is:

\begin{proposition}
 Given $K_1,K_2,K_3\in\Z^+$, if 
 $\mathcal{S}=\{\vec{v}_1,\vec{v}_2,\vec{v}_3\}\in\mathcal{O}_8(N,3)$
 with $N=K_1K_2K_3$
 and $K_j\mid\vec{v}_j$ for $1\le j\le 3$, 
 then $\mathcal{S}$ can be extended with another vector to a set in 
 $\mathcal{O}_8(N,4)$.
\end{proposition}

\begin{proof}
 Take $\vec{v}_{j0}=\vec{v}_j/K_j$. the vector $\vec{w}=\vec{v}_{10}\times\vec{v}_{20}\times\vec{v}_{30}$ is orthogonal to $\vec{v}_1$, $\vec{v}_2$ and $\vec{v}_3$. As these vectors are orthogonal, we have 
 $\|\vec{w}\|^2=\|\vec{v}_{10}\|^2\|\vec{v}_{20}\|^2\|\vec{v}_{30}\|^2$ that is 
 $N/K_1^2\cdot N/K_2^2\cdot N/K_3^2=N$.
\end{proof}

An example of the previous result with $K_1=12$, $K_2=15$, $K_3=20$, which corresponds to $N=3600$, is the orthogonal set $\{ \vec{v}_1, \vec{v}_2, \vec{v}_3\}$ with 
\[
 \begin{cases}
  \vec{v}_1=(12, -24, -12, 12, -24, 24, -36, 12 ),
  \\
  \vec{v}_2=(30, 15, -15, -15, -15, -30, 0, 30 ),
  \\
  \vec{v}_3=(40, 20, 20, 20, 20, 20, 0, 0).
 \end{cases}
\]
The vector $\vec{w}$ in the proof above is $( 2, 0, -33, -27, 26, 30, 9, 11)$. It has $\|\vec{w}\|^2=N$ and allows to extend the set with a new orthogonal vector. 
\medskip

There do not exist finite dimensional division algebras over $\R$ beyond the Cayley numbers. A more versatile extension of the quaternions $\mathcal{H}$ are the Clifford algebras, widely employed in theoretical physics \cite{woit}. In the following lines we explore how to traduce some instances of the orthogonality to the setting of some Clifford algebras in a constructive way, avoiding the reference to the general theory to minimize the prerequisites (see \cite{lounesto} for a basic approach to Clifford algebras, mainly through Euclidean examples, and \cite{garling} for a more advanced introduction).
\smallskip

Let $\mathbb{F}_2$ be the field of two elements and $V$ the subspace of codimension~1 of $\mathbb{F}_2^n$ defined by $x_1+\dots+x_n=0$.  For our purposes it will be convenient to define  the function
\[
 s:V\longmapsto \mathbb{F}_2,
 \qquad\text{where}\quad
 s(\vec{v})=\frac 12 \#\{1\le j\le n\,:\; v_j=1\}\pmod{2}.
\]
Note that it is well defined because each vector in $V$ has an even number of ones. 

The even subalgebra $\mathcal{E}_n$ of the Clifford algebra
$\mathcal{C}\ell_{0,n}(\R)$
can be defined as having a basis $\mathcal{B}=\{\mathbf{e}_{\vec{a}}\,:\, \vec{a}\in V\}$ over $\R$ 
obeying the algebra operation \cite[\S2.13]{lounesto}
\begin{equation}\label{clifford}
 \mathbf{e}_{\vec{a}}\mathbf{e}_{\vec{b}}
 =
 S(\vec{a},\vec{b})\mathbf{e}_{\vec{a}+\vec{b}}
 \qquad\text{with}\quad 
 S(\vec{a},\vec{b})
 =
 (-1)^{\sum_{j=1}^n\sum_{k=1}^ja_jb_k}.
\end{equation}
The element 
$\mathbf{e}_{\vec{0}}$
is the unit and $\dim_{\R}\mathcal{E}_n=2^{n-1}$. 

Consider the natural bijective map 
$\phi:\mathcal{E}_n\longmapsto\R^{2^{n-1}}$
assigning to each element of 
$\mathcal{E}_n$ its coordinates in the basis $\mathcal{B}$.
The following result relates the multiplication by the basis elements to the orthogonality.

\begin{proposition}\label{cliforth}
 Consider a set $V_0\subset V$ of cardinality $n_0$ such that 
 $s(\vec{u})+s(\vec{v})+\vec{u}\cdot\vec{v}$ 
 is odd for any distinct $\vec{u},\vec{v}\in V_0$.
 Then given $\mathbf{e}\in\mathcal{E}_n$ with integral coordinates and norm $\sqrt{N}$, we have  
 $\big\{\phi(\mathbf{e}\mathbf{e}_{\vec{v}})\,:\, \vec{v}\in V_0\big\}\in\mathcal{O}_{2^{n-1}}(N,n_0)$. 
\end{proposition}

In $\mathcal{E}_3$
we can identify 
$\mathbf{e}_{(0,0,0)}=1$, 
$\mathbf{e}_{(1,1,0)}=\mathbf{i}$, 
$\mathbf{e}_{(1,0,1)}=\mathbf{j}$, 
$\mathbf{e}_{(0,1,1)}=\mathbf{k}$
preserving the algebra operations \eqref{clifford}. In this way,
$\mathcal{E}_3$ becomes isomorphic to the algebra of quaternions $\mathcal{H}$. 
The orthogonality of the vectors defined by the coordinates of $\{\mathbf{q}, \mathbf{q}\mathbf{i}, \mathbf{q}\mathbf{j}, \mathbf{q}\mathbf{k}\}$, what was used in the proof of Theorem~\ref{case4} (the ordering is not important by conjugation), is then covered by Proposition~\ref{cliforth} choosing $V_0=V$. 

It is unclear if it is possible to recover in this context the matrix $C(\vec{x})$ associated to the multiplication by octonion units. Probably the underlying difficulty is that Clifford algebras are associative and the algebra of Cayley numbers is not. 

\begin{lemma}\label{ssclif}
 We have 
 $s(\vec{v})=\sum_{j=1}^n\sum_{k=1}^jv_jv_k$ 
 for every $\vec{v}\in V$. 
\end{lemma}

\begin{proof}
 Let us consider the coordinates of $\vec{v}$ as integers in $\{0,1\}$. 
 We know $(v_1+\dots +v_n)^2\equiv 0\pmod{4}$ because $\vec{v}\in V$.
 Expanding the square and using $v_j^2= v_j$, we obtain 
 \[
  \sum_{j=1}^n v_j 
  +
  2
  \sum_{j=1}^n\sum_{k=1}^{j-1}v_jv_k
  =
  -\sum_{j=1}^n v_j 
  +
  2
  \sum_{j=1}^n\sum_{k=1}^jv_jv_k
  \equiv 0\pmod{4}.
 \]
 Dividing by $2$, we get the result.
\end{proof}

\begin{proof}[Proof of Proposition~\ref{cliforth}]
 Let 
 $\mathbf{e}=\sum_{\vec{a}\in V}\lambda_{\vec{a}}\mathbf{e}_{\vec{a}}$
 with $\lambda_{\vec{a}}\in\Z$.
 Using \eqref{clifford}, the $\mathbf{e}_{\vec{c}}$ coordinates of 
 $\mathbf{e}\mathbf{e}_{\vec{v}}$
 and
 $\mathbf{e}\mathbf{e}_{\vec{u}}$
 are
 $
  S(\vec{c}-\vec{v},\vec{v})\lambda_{ \vec{c}-\vec{v}}
 $ and  
 $
  S(\vec{c}-\vec{u},\vec{u})\lambda_{ \vec{c}-\vec{u}}
 $, respectively.
 In the same way, noting $+1=-1$ in $\mathbb{F}_2$, the $\mathbf{e}_{\vec{c}+\vec{v}-\vec{u}}$ coordinates are
 $
  S(\vec{c}-\vec{u},\vec{v})\lambda_{ \vec{c}-\vec{u}}
 $
 and
 $
  S(\vec{c}-\vec{v},\vec{u})\lambda_{ \vec{c}-\vec{v}}
 $.
 Then the orthogonality of 
 $\phi(\mathbf{e}\mathbf{e}_{\vec{v}})$ 
 and
 $\phi(\mathbf{e}\mathbf{e}_{\vec{u}})$ 
 follows if 
 \[
  S(\vec{c}-\vec{v},\vec{v})
  S(\vec{c}-\vec{u},\vec{u})
  =
  -
  S(\vec{c}-\vec{u},\vec{v})
  S(\vec{c}-\vec{v},\vec{u})
 \]
 because in this case the contribution to the scalar product of the coordinates indexed with $\vec{c}$ and $\vec{c}+\vec{v}-\vec{u}$ cancel. 
 Recalling the definition of $S$ in \eqref{clifford} and noting that the exponent is a bilinear form, the previous equation translates into 
 \[
  -\sum_{j=1}^n\sum_{k=1}^j
  \big(
  v_jv_k
  +u_ju_k
  \big)
  =
  -\sum_{j=1}^n\sum_{k=1}^j
  \big(
  u_jv_k
  +v_ju_k
  \big)
  +1
  \qquad\text{in }\mathbb{F}_2.
 \]
 For $\vec{u},\vec{v}\in V$ we have
 \[
  0
  =
  \Big(
  \sum_{j=1}^n
  u_j 
  \Big)
  \Big(
  \sum_{k=1}^n
  v_k 
  \Big)
  =
  \sum_{j=1}^n\sum_{k=1}^j
  u_jv_k
  +
  \sum_{j=1}^n\sum_{k=1}^j
  u_kv_j
  -
  \vec{u}\cdot\vec{v}.
 \]
 Then the previous relation reads 
 \[
  -\sum_{j=1}^n\sum_{k=1}^j
  \big(
  v_jv_k
  +u_ju_k
  \big)
  +  
  \vec{u}\cdot\vec{v}
  =1
 \]
 and the result follows from Lemma~\ref{ssclif}. 
\end{proof}

For instance, for $n=5$, which corresponds to $d=2^{n-1}=16$, a valid set in Proposition~\ref{cliforth} is
\[
 V_0= \big\{\vec{0}, (0,0,1,0,1), (0,0,1,1,0), (0,1,1,0,0), (1,0,1,0,0)\big\}.
\]
Working out the coordinates of $\phi(\mathbf{e}\mathbf{e}_{\vec{v}})$ with $\mathbf{e}$ an arbitrary element with $\phi(\mathbf{e})=(x_0,\dots, x_{15})$ we obtain the rows of a matrix $A=(A_1|A_2)$ with 
\[
A_1=
 \begin{pmatrix}
x_{0}  & x_{1}  & x_{2}  & x_{3}  & x_{4}  & x_{5}  & x_{6}  & x_{7}  \\
-x_{2} & -x_{3}  & x_{0}  & x_{1}  & x_{6}  & x_{7} & -x_{4} & -x_{5}  \\
-x_{3}  & x_{2} & -x_{1}  & x_{0} & -x_{7}  & x_{6} & -x_{5}  & x_{4} \\
-x_{6} & -x_{7}  & x_{4}  & x_{5} & -x_{2} & -x_{3}  & x_{0}  & x_{1} \\
-x_{10} & -x_{11}  & x_{8}  & x_{9}  & x_{14}  & x_{15} & -x_{12} & -x_{13}
\end{pmatrix}
\]
and
\[
A_2=
 \begin{pmatrix}
 x_{8}  & x_{9}  & x_{10}  & x_{11}  & x_{12}  & x_{13}  & x_{14}  & x_{15}\\
 x_{10}  & x_{11} & -x_{8} & -x_{9} & -x_{14} & -x_{15}  & x_{12}  & x_{13}\\
-x_{11}  & x_{10} & -x_{9}  & x_{8} & -x_{15}  & x_{14} & -x_{13}  & x_{12}\\
-x_{14} & -x_{15}  & x_{12}  & x_{13} & -x_{10} & -x_{11}  & x_{8}  & x_{9}\\
-x_{2} & -x_{3}  & x_{0}  & x_{1}  & x_{6}  & x_{7} & -x_{4} & -x_{5}
\end{pmatrix}.
\]
According to Proposition~\ref{cliforth}, the  rows of $A$ are orthogonal of the same norm and hence $\mathcal{C}_{16}(1,5)=\Z^+$. This a little discouraging because an exhaustive search shows that for $d=16$ the cardinality of $V_0$ is at most~5 and Corollary~\ref{triv84} gives a better result in this case.


\begin{thebibliography}{10}

\bibitem{baez}
J.~C. Baez.
\newblock The octonions.
\newblock {\em Bull. Amer. Math. Soc. (N.S.)}, 39(2):145--205, 2002.

\bibitem{BaGr}
P.~T. Bateman and E.~Grosswald.
\newblock Positive integers expressible as a sum of three squares in
  essentially only one way.
\newblock {\em J. Number Theory}, 19(3):301--308, 1984.

\bibitem{CoPe}
B.~Coan and C.-t. Perng.
\newblock Factorization of {H}urwitz quaternions.
\newblock {\em Int. Math. Forum}, 7(41-44):2143--2156, 2012.

\bibitem{CoSm}
J.~H. Conway and D.~A. Smith.
\newblock {\em On quaternions and octonions: their geometry, arithmetic, and
  symmetry}.
\newblock A K Peters, Ltd., Natick, MA, 2003.

\bibitem{cremona}
J.~Cremona.
\newblock Letter to the editor.
\newblock {\em Amer. Math. Monthly}, 94:757--758, 1987.

\bibitem{DiZh}
J.~Ding and A.~Zhou.
\newblock Eigenvalues of rank-one updated matrices with some applications.
\newblock {\em Appl. Math. Lett.}, 20(12):1223--1226, 2007.

\bibitem{garling}
D.~J.~H. Garling.
\newblock {\em Clifford algebras: an introduction}, volume~78 of {\em London
  Mathematical Society Student Texts}.
\newblock Cambridge University Press, Cambridge, 2011.

\bibitem{GaLa}
L.~N. Gatti and J.~Lacalle.
\newblock A model of discrete quantum computation.
\newblock {\em Quantum Inf. Process.}, 17(8):Paper No. 192, 18, 2018.

\bibitem{grosswald}
E.~Grosswald.
\newblock {\em Representations of integers as sums of squares}.
\newblock Springer-Verlag, New York, 1985.

\bibitem{KaSo}
I.~L. Kantor and A.~S. Solodovnikov.
\newblock {\em Hypercomplex numbers}.
\newblock Springer-Verlag, New York, 1989.
\newblock An elementary introduction to algebras, Translated from the Russian
  by A. Shenitzer.

\bibitem{LaGa}
J.~Lacalle and L.~N. Gatti.
\newblock Discrete quantum computation and {L}agrange's four-square theorem.
\newblock {\em Quantum Inf. Process.}, 19(1):Paper No. 34, 20, 2020.

\bibitem{lounesto}
P.~Lounesto.
\newblock {\em Clifford algebras and spinors}, volume 239 of {\em London
  Mathematical Society Lecture Note Series}.
\newblock Cambridge University Press, Cambridge, 1997.

\bibitem{massey}
W.~S. Massey.
\newblock Cross products of vectors in higher-dimensional {E}uclidean spaces.
\newblock {\em Amer. Math. Monthly}, 90(10):697--701, 1983.

\bibitem{spira}
R.~Spira.
\newblock The {D}iophantine equation {$x^{2}+y^{2}+z^{2}=m^{2}$}.
\newblock {\em Amer. Math. Monthly}, 69:360--364, 1962.

\bibitem{stillwell}
J.~Stillwell.
\newblock {\em Naive {L}ie theory}.
\newblock Undergraduate Texts in Mathematics. Springer, New York, 2008.

\bibitem{woit}
P.~Woit.
\newblock {\em Quantum theory, groups and representations}.
\newblock Springer, Cham, 2017.
\newblock An introduction.

\bibitem{zvengrowski}
P.~Zvengrowski.
\newblock A {$3$}-fold vector product in {$R^{8}$}.
\newblock {\em Comment. Math. Helv.}, 40:149--152, 1966.

\end{thebibliography}

\end{document}